\documentclass{amsart}
\numberwithin{equation}{section} 
\usepackage{graphicx,tabularx,array}
\usepackage{amsfonts}
\usepackage{amsthm}
\usepackage{amsmath}
\usepackage{amssymb}
\usepackage{amscd}
\usepackage{mathrsfs}
\usepackage{verbatim}
\usepackage{xypic}
\title
	[BHK Mirrors and Shioda Maps]
	{Berglund-H\"ubsch-Krawitz Mirrors via Shioda Maps}
\newtheorem{thm}{Theorem}[section]
\newtheorem{cor}[thm]{Corollary}
\newtheorem{lem}[thm]{Lemma}
\newtheorem{prop}[thm]{Proposition}

\newtheorem{rem}[thm]{Remark}

\newcommand{\RR}{\mathbb{R}}
\newcommand{\CC}{\mathbb{C}}
\newcommand{\ZZ}{\mathbb{Z}}
\newcommand{\PP}{\mathbb{P}}

\newcommand{\NN}{\mathbb{N}}

\newcommand{\aut}{\text{Aut}}

\newcommand{\diag}{\text{diag}}

\begin{document}

\title{Berglund-H\"ubsch-Krawitz Mirrors via Shioda Maps}

\author{Tyler L. Kelly}

\email{tykelly@math.upenn.edu}
\address{Department of Mathematics, University of Pennsylvania, 209 S. 33rd St, Philadelphia, PA 19103}
\thanks{This material is based upon work supported by the National Science Foundation under Grant No. DGE-1321851.}


\date{April 11, 2013}



\begin{abstract} We give an elementary approach to proving the birationality of multiple Berglund-H\"ubsch-Krawitz (BHK) mirrors by using Shioda maps. We do this by creating a birational picture of the BHK correspondence in general. Although a similar result has been obtained in recent months by Shoemaker, our proof is new in that it sidesteps using toric geometry and drops an unnecessary hypothesis. We give an explicit quotient of a Fermat variety to which the mirrors are birational. \end{abstract}
\maketitle

\section{Introduction}

The mirror symmetry conjecture predicts that for a Calabi-Yau variety, $M$, there exists another Calabi-Yau variety, $W$, so that various geometric and physical data is exchanged between $M$ and $W$.  A classical relationship found between so-called mirror pairs is that on the level of cohomology
$$
H^{p,q}(M, \CC) \cong H^{N-p,q}(W, \CC),
$$
provided that both Calabi-Yau varieties $M$ and $W$ are $N$-dimensional.  In 1992, Berglund and H\"ubsch proposed such a mirror symmetry relationship between finite quotients of hypersurfaces in weighted projective $n$-space \cite{BH}.  
  Suppose $F_A$ is a polynomial, 
\begin{equation}
F_A = \sum_{i=0}^n \prod_{j=0}^n x_j^{a_{ij}},
\end{equation}
where $a_{ij} \in \NN$, so that there exist positive integers $q_j$ and $d$ so that $\sum_{j} a_{ij} q_j=d$ for all $i$ ( i.e., $F_A$ is quasihomogeneous). The polynomial $F_A$ cuts out a hypersurface $X_A := Z(F_A) \subset W\PP^n(q_0, \ldots, q_n)$ of dimension $N=n-1$.  Further assume that this hypersurface is a quasi-smooth Calabi-Yau variety (the Calabi-Yau condition is equivalent to $\sum_i q_i = d$ and see Section 2 for details about the quasismooth condition). Greene and Plesser proposed a mirror to $X_A$ when the polynomial $F_A$ was Fermat \cite{GP}.  Their proposed mirror for the hypersurface $X_A$ was a quotient of $X_A$ by all its phase symmetries of $X_A$ leaving the cohomology $H^{n,0}(X_A)$ invariant. The problem was that their proposal does not work well for the case when $X_A$ was not a Fermat hypersurface. Berglund and H\"ubsch proposed that the mirror of the hypersurface $X_A$ should relate to a hypersurface $X_{A^T}$ cut out by 
\begin{equation}
F_{A^T} =  \sum_{i=0}^n \prod_{j=0}^n x_j^{a_{ji}}.
\end{equation}
The hypersurface $X_{A^T}$ sits inside a different weighted-projective 4-space, $W\PP^n(r_0, \ldots, r_n)$.  Berglund and H\"ubsch proposed that the mirror of $X_A$ should be a quotient of this new hypersurface $X_{A^T}$ by a suitable subgroup $P$ of the phase symmetries.  In several examples, they showed that  $X_A$ and $X_{A^T}/P$ satisfy the classical mirror symmetry relation in that 
$$
h^{p,q} (X_A, \CC) = h^{n-1-p, q} (X_{A^T}/P, \CC).
$$

This proposal fell out of favor when Batyrev and Borisov developed the powerful toric approach (see \cite{Ba}, \cite{BB1}, and \cite{BB2}). In the 2000s, Krawitz revived Berglund and H\"ubsch's proposal by giving a rigorous mathematical description of their mirror and proving a mirror symmetry theorem on the level of Frobenius algebra structures \cite{Kr}.  

Krawitz also generalized the Berglund-H\"ubsch mirror proposal by introducing the notion of a dual group:  We start with a polynomial $F_A$. Consider the group $SL(F_A)$ of phase symmetries of $F_A$ leaving $H^{n,0}(X_A)$ invariant. Define the subgroup $J_{F_A}$ of $SL(F_A)$ to be the group consisting of the phase symmetries induced by the $\CC^*$ action on weighted-projective space (so that all elements of $J_{F_A}$ act trivially on the weighted-projective space). Take the group $G$ to be some subgroup of $SL(F_A)$ containing $J_{F_A}$, i.e., $J_{F_A} \subseteq G \subseteq SL(F_A)$.  We obtain a Calabi-Yau orbifold $Z_{A,G} := X_A/ \tilde G$ where $\tilde G := G/J_{F_A}$.   Consider the analogous groups $SL(F_{A^T})$ and $J_{F_{A^T}}$ for the polynomial $F_{A^T}$.  Krawitz defined the dual group $G^T$ relative to $G$ so that $J_{F_{A^T}} \subseteq G^T\subseteq SL(F_{A^T})$.  For precise definitions of these groups, we direct the reader to Section 2. Take the quotient $\tilde G^T:= G^T/ J_{F_{A^T}}$. The Berglund-H\"ubsch-Krawitz mirror to the orbifold $Z_{A,G} $ is the orbifold $Z_{A^T,G^T} := X_{A^T}/ \tilde G^T$.  Chiodo and Ruan proved the classical mirror symmetry statement for the mirror pair $Z_{A,G}$ and $Z_{A^T,G^T}$ is satisfied on the level of Chen-Ruan cohomology \cite{CR}:
\begin{equation}
H^{p,q}_{\text{CR}} ( Z_{A,G} , \CC) \cong H^{n-1-p,q}_{\text{CR}}(Z_{A^T, G^T}, \CC).
\end{equation}

One can compare the mirrors found in Berglund-H\"ubsch-Krawitz (BHK) mirror duality to the mirrors of Batyrev and Borisov. In Batyrev-Borisov mirror symmetry, a family of Calabi-Yau hypersurfaces in one toric variety all have mirrors that live inside a family of hypersurfaces in a different toric variety. A feature of BHK mirror symmetry is that it proposes possibly distinct mirrors of isolated points of the family in the Calabi-Yau moduli space--not mirrors of families like the work of Batyrev and Borisov.  These BHK mirrors of the isolated points may not live in the same family.  Suppose one starts with two quasihomogeneous potentials $F_A$ and $F_{A'}$ 
\begin{equation}
F_A = \sum_{i=0}^n \prod_{j=0}^n x_j^{a_{ij}}; \qquad F_{A'} = \sum_{i=0}^n \prod_{j=0}^n x_j^{a_{ij}'}.
\end{equation}
Assume that there exist positive integers $q_i, q_i'$ so that $X_A = Z(F_A) \subseteq W\PP^n(q_0, \ldots, q_n)$ and $X_{A'} = Z(F_{A'}) \subseteq W\PP^n(q_0', \ldots, q_n')$ and that $X_A$ and $X_{A'}$ are Calabi-Yau.  Take $G$ and $G'$ to be subgroups of the group of phase symmetries that leave  the respective cohomologies $H^{n,0}(X_A, \CC)$ and $H^{n,0}(X_{A'},\CC)$ to be invariant. We obtain two Calabi-Yau orbifolds $Z_{A,G}$ and $Z_{A', G'}$. One can find examples $Z_{A, G}$ and $Z_{A', G'}$ in the same family where their BHK mirrors $Z_{A^T, G^T}$ and $Z_{(A')^T, (G')^T}$ will be quotients of hypersurfaces in different weighted-projective spaces. See Section 5 for an explicit example. 

Since the mirrors proposed by BHK and Batyrev-Borisov mirror symmetry are different, we ask the question of how we can relate them.  Iritani suggested to look at the birational geometry of the mirrors $Z_{A^T, G^T}$ and $Z_{(A')^T, (G')^T}$. Shoemaker proved that when the hypersurfaces $X_A$ and $X_{A'}$ are in the same weighted-projective space and the groups $G$ and $G'$ are equal, then the BHK mirrors $Z_{A^T, G^T}$ and $Z_{(A')^T, (G')^T}$ are birational. He proves this claim by using a reinterpretation of BHK duality into toric language \cite{Sho}.  In this paper, we drop one of these hypotheses and prove the following theorem.

\begin{thm}\label{main}
Let $Z_{A,G}$ and $Z_{A',G'}$ be Calabi-Yau orbifolds as above. If the groups $G$ and $G'$ are equal, then the BHK mirrors $Z_{A^T, G^T}$ and $Z_{(A')^T, (G')^T}$ of these orbifolds are birational.
\end{thm}

This theorem drops a hypothesis from Shoemaker's theorem in that the hypersurfaces $X_A$ and $X_{A'}$ need not be in the same weighted-projective space.  It is unclear if this is a strict generalization of Shoemaker's result as no examples of a common group of phase symmetries $G$ are known when $X_A$ and $X_{A'}$ are hypersurfaces in different weighted-projective spaces; however, the proof we present is novel in that it does not use any toric geometry but rather uses rational maps known as Shioda maps. Using Shioda maps, we show that the Calabi-Yau orbifolds $Z_{A,G}$, $Z_{A',G'}$, $Z_{A^T, G^T}$, and $Z_{(A')^T, (G')^T}$ are all birational to different finite group quotients of a Fermat variety in projective space $\PP^n$.  In the case where the group $G$ equals $G'$, we see that the mirrors $Z_{A^T, G^T}$ and $Z_{(A')^T, (G')^T}$ are birational to the same quotient of the Fermat variety.

Originally, Shioda used these maps to compute Picard numbers of Delsarte surfaces in \cite{Shi}.  These maps entered the multiple mirror literature in \cite{BvGK} where they were generalized and then used to investigate Picard-Fuchs equations of different pencils of quintics in $\PP^4$. The Shioda maps were then further generalized to look at GKZ hypergeometric systems for certain families of Calabi-Yau varieties in weighted-projective space in \cite{Bi}.  This paper provides a more concrete description of how Shioda maps relate to BHK mirror symmetry than the previous two papers, and explains the groups used in the theorems of \cite{BvGK} and \cite{Bi} in the context of BHK mirrors (see Section 3.2).  In future work, this framework will be used to probe K\"ahler moduli space of the Calabi-Yau orbifolds.

{\bf Organization of the Paper.} In Section 2, we review the BHK mirror construction and the results of Chiodo and Ruan.  In Section 3, we use the Shioda map to discuss the birational geometry of the BHK mirrors and the groups involved in \cite{BvGK} and \cite{Bi}. We then show the birationality of the Calabi-Yau orbifolds $Z_{A,G}$ and $Z_{A^T, G^T}$ to finite quotients of a Fermat variety in projective space.  In Section 4, we use the results found in Section 3 to prove Theorem \ref{main}. Section 5 concludes the paper by giving an explicit example where we take two Calabi-Yau orbifolds and show that their BHK mirrors are hypersurfaces of different quotients of weighted projective spaces. We then show that the BHK mirrors provided in the example are birational. 

 {\bf Acknowledgments.}  The author would like to thank Charles Siegel, Ting Chen, Alberto Garc\'ia-Raboso, Antonella Grassi, and Tony Pantev for their enlightening conversations on the subject.   He would like to give special thanks to his advisor, Ron Donagi, for the support, mentoring, conversations and introducing him to the Berglund-H\"ubsch literature.  This work was done under the support of a National Science Foundation Graduate Research Fellowship.


\section{Berglund-H\"ubsch-Krawitz Duality}

We start with a matrix $A$ with nonnegative integer entries $(a_{ij})_{i,j=0}^n$. Define a polynomial
$$
F_A = \sum_{i=0}^n \prod_{j=0}^n x_j^{a_{ij}}
$$
and impose the following conditions:

\begin{enumerate}

\item the matrix $A$ is invertible;

\item the polynomial $F_A$ is quasihomogeneous, i.e., there exist positive integers $q_j, d$  so that 
$$
\sum_{j=0}^n a_{ij} q_j  = d, 
$$
for all $i$; and

\item the polynomial $F_A$ is a non-degenerate potential away from the origin, i.e., we are assuming that, when viewing $F_A$ as a polynomial in $\CC^{n+1}$, $Z(F_A)$ has exactly one singular point (at the origin).  

  \end{enumerate}

\begin{rem}These conditions are restrictive. By Theorem 1 of \cite{KS}, there is a classification of such polynomials. That is, $F_A$ can be written as a sum of invertible potentials, each of which must be of one of the three so-called \emph{atomic types}:

\begin{equation}\begin{aligned}
W_{\text{Fermat}} &:= x^a, \\
W_{\text{loop}} &:= x_1^{a_1}x_2 + x_2^{a_2}x_3 + \ldots +x_{m-1}^{a_{m-1}}x_m + x_m^{a_m}x_1, \text{ and } \\
W_{\text{chain}} &:=x_1^{a_1}x_2 + x_2^{a_2}x_3 + \ldots x_{m-1}^{a_{m-1}}x_m + x_m^{a_m}.
\end{aligned}\end{equation}
\end{rem}

 Using Condition (1), we define the matrix $B = dA^{-1}$, where $d$ is a positive integer so that all the entries of $B$ are integers (note that $d$ is not necessarily the smallest such $d$). Take $e := (1, \ldots, 1)^T \in \RR^n$ and 
\[
q := Be\text{, i.e., }q_i = \sum_j b_{ij}.
\]
Then the polynomial $F_A$ defines a zero locus $X_A = Z(F_A) \subseteq W\PP^n(q_0, \ldots, q_n)$.  Indeed, with these weights, the polynomial $F_A$ is quasihomogeneous:  each monomial in $F_A$ has degree $\sum_{j=0}^n a_{ij}q_j = d$, as $Aq = ABe = de$.  Condition (2) above is used to ensure that each integer $q_i$ is positive.  

Assume further that $\sum_i q_i = d$ is the degree of the polynomial, which implies that the hypersurface $X_A$ is a Calabi-Yau variety.  Define $\text{Sing}(V)$ to be the singular locus of any variety $V$, we say the hypersurface $X_A$ is quasi-smooth if $\text{Sing}(X_A) \subseteq \text{Sing}(W\PP^n(q_0,\ldots, q_n)) \cap X_A$.  Condition (3) above implies that our hypersurface $X_A$ is quasi-smooth.  We remark that Condition (1) is used once again when we introduce the BHK mirror in Section 2.2: it ensures that the matrix $A^T$ is a matrix of exponents of a polynomial with $n+1$ monomials and $n+1$ variables.

\subsection{Group of Diagonal Automorphisms}
Let us discuss the groups of symmetries of the Calabi-Yau variety $X_{F_A}$.  Firstly, consider the scaling automorphisms of $\CC^{n+1} \setminus \{0\}$, $n\geq 2$.  There is a subgroup, $(\CC^*)^{n+1}$, of the automorphisms of $\CC^{n+1} \setminus \{0\}$.  Explicitly, an element $ (\lambda_0,\ldots, \lambda_n) \in (\CC^*)^{n+1}$ acts on any element $\mathbf{x} = (x_0,\ldots, x_n) \in \CC^{n+1} \setminus \{0\}$ by:
$$
(\lambda_0,\ldots, \lambda_n) \times (x_0,\ldots, x_n)  \longmapsto  (\lambda_0x_0,\ldots, \lambda_n x_n).
$$
We view the weighted projective $n$-space $W\PP^{n}(q_0,\ldots, q_n)$ as a quotient of $\CC^{n+1}\setminus \{0\}$ by a subgroup $\CC^* \subset (\CC^*)^{n+1}$ consisting of the elements that can be written $(\lambda^{q_0/d}, \ldots, \lambda^{q_n/d}) $ for some $\lambda \in \CC^*$.

Moreover, there is a second subgroup of $(\CC^*)^{n+1}$, denoted $\aut(F_A)$, which can be defined as 
\[
\aut(F_A) := \{ (\lambda_0, \ldots, \lambda_n) \in (\CC^*)^{n+1} | F_A(\lambda_0 x_0, \ldots, \lambda_n x_n) = F_A(x_0, \ldots, x_n)\text{ for all } (x_0,\ldots, x_n)\}.
\]
This group is sometimes referred to as the group of diagonal automorphisms or the group of scaling symmetries. Note that for $(\lambda_0,\ldots, \lambda_n)$ to be an element of $\aut(F_A)$, each monomial $\prod_{j=0}^n x_j^{a_{ij}}$ must be invariant under the action of $(\lambda_0, \ldots, \lambda_n)$. 

Using the classification of Kreuzer and Skarke (see Remark 2.1), we can see that for any polynomial of one of the atomic types that each $\lambda_i$ must have modulus 1.  If the polynomial $F_A$ is of fermat-type, then $\lambda^ax^a = x^a$ hence $\lambda^a = 1$.  If $F_A$ is of loop-type, then $\lambda_i^{a_i}\lambda_{i+1} = 1$ for all $i < a_m$, hence $\lambda_{i+1} = \lambda_i^{-a_i}$. Moreover, $\lambda_m^{a_m} \lambda_1 = 1$ hence $\lambda_1 = \lambda_m^{-a_m} = \lambda_{m-1}^{a_ma_{m-1}} = \cdots = \lambda_1^{(-1)^ma_1\cdots a_m}$. If $|\lambda| \neq 1$ then $(-1)^m a_1\ldots a_m = 1$. This would require $m$ to be even and $a_i$ to be $1$ for all $i$. However, then the degree of the polynomial, $d$, must be $q_1 + q_2$; however $d = \sum_{i=0}^n q_i$, $n \geq 2$, and $q_i >0$, hence a contradiction is reached.
Lastly, if $F_A$ is of chain-type, $\lambda_m^{a_m}x_m^{a_m} = x_m^{a_m}$, hence $|\lambda_m|^{a_m} = 1$. This implies that $|\lambda_{m-1}^{a_{m-1}}\lambda_m| = |\lambda_{m-1}^{a_{m-1}}| = 1$, and so on, hence $|\lambda_i| = 1$. Any polynomial that is a combination of such types has an analogous argument.

Since each $\lambda_i$ can be written as  $e^{i\theta_i}$, for some $\theta_i \in \RR$, we can then see that $(\lambda_0,\ldots, \lambda_n) \in \aut(F_A)$ if and only if we have that $\prod_{j=0}^n e^{ia_{ij}\theta_j} = 1$ for all $i$. The map $(\lambda_0, \ldots, \lambda_n) \mapsto (\frac{1}{2\pi i} \log (\lambda_0), \ldots, \frac{1}{2\pi i} \log(\lambda_n))$ induces an isomorphism
\begin{equation}\label{autisom}
\aut(F_A) \cong \left\{ (z_0,\ldots, z_n) \in (\RR/\ZZ)^n\text{ } \middle| \text{ }A \left(\begin{array}{c}z_0 \\ \vdots \\ z_n\end{array}\right) \in \ZZ^{n+1} \right\}.
\end{equation}
  We then observe that we can describe $\aut(F_A)$ as being generated by the elements 
\[
\rho_i = (e^{2\pi i b_{0i}/d}, \ldots, e^{2\pi b_{ni}/d}) \in (\CC^*)^{n+1}.
\]
Moreover, there is a characterization by Artebani, Boissi\`ere, and Sarti of the group $\aut(F_A)$ (Proposition 2 of \cite{ABS}):
\begin{prop}
$\emph{Aut}(F_A)$ is a finite abelian group of order $|\det A|$.  If we think of $F_A$ as a sum of atomic types, $F_{A_1}(x_0, \ldots, x_{i_1}) + \ldots + F_{A_k}(x_{i_{k-1}+1}, \ldots, x_n)$, then we may characterize the elements of $\emph{Aut}(F_A)$ as being the product of the $k$ groups $\emph{Aut}(F_{A_i})$. The groups $\emph{Aut}(F_{A_i})$ are determined based on the atomic types:
\begin{enumerate}
\item For a summand of Fermat type $W_{\text{Fermat}} = x^a$, the group $\emph{Aut}(W_{\text{Fermat}})$ is isomorphic to  $\ZZ/a\ZZ$ and generated by $\varphi = e^{2\pi i / a} \in \CC^*$.
\item For a summand of loop type $W_{\text{loop}} = x_1^{a_1}x_2 + x_2^{a_2}x_3 + \ldots +x_{m-1}^{a_{m-1}}x_m + x_m^{a_m}x_1$, the group $\emph{Aut}(W_{\text{loop}})$ is isomorphic to $ \ZZ / \Gamma\ZZ$ where $\Gamma =a_1\cdots a_m + (-1)^{m+1}$  and  generated by $(\varphi_1, \ldots, \varphi_m)\in (\CC^*)^m$, where 
\[
\varphi_1 : = e^{2\pi i (-1)^m/\Gamma}, \text{ and } \varphi_i : = e^{2\pi i (-1)^{m+1-i} a_1\cdots a_{i-1}/ \Gamma}, i \geq 2.
\]
\item For a summand of chain type, $W_{\text{chain}} = x_1^{a_1}x_2 + x_2^{a_2}x_3 + \ldots x_{m-1}^{a_{m-1}}x_m + x_m^{a_m}$, the group $\emph{Aut}(W_{\text{chain}} )$ is isomorphic to $ \ZZ/ (a_1 \cdots a_m) \ZZ$, and generated by $(\varphi_1, \ldots, \varphi_m) \in (\CC^*)^m$, where 
\[
\varphi_i = e^{2\pi i (-1)^{m+i}/a_i\cdots a_m}.
\]
\end{enumerate}
\end{prop}

Note that there is some overlap in the subgroups of $(\CC^*)^{n+1}$. Let $J_{F_A}:= \aut(F_A) \cap \CC^*$.  The group $J_{F_A}$ is generated by $(e^{2\pi i q_0/d}, \ldots, e^{2\pi i q_n/d})$, which is clearly in $\aut(F_A)$ because $\sum_{j=0}^n a_{ij}q_j = d$ and the alternate description provided by the isomorphism above in Equation \ref{autisom} (moreover, $(e^{2\pi i q_0/d}, \ldots, e^{2\pi i q_n/d}) = \prod_{i=0}^n \rho_i \in \aut(F_A)$).

We now introduce the group
\[
SL(F_A) := \left\{ (\lambda_0, \ldots, \lambda_n) \in \aut(F_A) \middle| \prod_{j=0}^n \lambda_j = 1\right\}.
\]
The group $J_{F_A}$ is a subgroup of $SL(F_A)$ as a generator of $J_{F_A}$ is the element $(e^{2\pi i q_j/d})_j$ and $\prod_j e^{2\pi i q_j / d} = e^{\frac{2\pi i}{d} \sum_j q_j} = 1$. Fix a group $G$ so that $J_{F_A} \subseteq G \subseteq SL(F_A)$ and put $\tilde G := G/ J_{F_A}$.  To help summarize, we have the following diagram of groups:

\centerline{
\xymatrix{
	J_{F_A} \ar@{^{(}->}[d]\ar@{^{(}->}[r]& J_{F_A} \ar@{^{(}->}[d]\ar@{^{(}->}[r] &J_{F_A} \ar@{^{(}->}[d]\ar@{^{(}->}[r]& \CC^* \ar@{^{(}->}[d]
	\\
	G \ar@{^{(}->}[r] \ar[d]& SL(F_A)\ar@{^{(}->}[r] & \aut(F_A) \ar@{^{(}->}[r] & (\CC^*)^{n+1}\ar[d]
	\\
	\tilde{G} : = G/J_{F_A}&&& (\CC^*)^{n+1}/\CC^*
	}
}
  
Consider the Calabi-Yau orbifold, $Z_{A,G} : = X_{F_A}/\tilde G\subset W\PP^n(q_0, \ldots, q_n) / \tilde G$. We now will describe the Berglund-H\"ubsch-Krawitz mirror to it.
 
 \subsection{The Berglund-H\"ubsch-Krawitz Mirror}

In this section, we construct the BHK mirror to the Calabi-Yau orbifold $Z_{A,G}$ defined above.  Take the polynomial
\begin{equation}
F_{A^T} = \sum_{i=0}^n \prod_{j=0}^n X_j^{a_{ji}}.
\end{equation}
It is quasihomogeneous because there exist positive integers $r_i :=\sum_j b_{ji}$  so that 
\begin{equation}
F_{A^T}(\lambda^{r_0} X_0 , \ldots, \lambda^{r_n} X_n) = \lambda^{d}F_{A^T}(X_0 , \ldots,  X_n).
\end{equation}
Note that the polynomial $F_{A^T}$ cuts out a well-defined Calabi-Yau hypersurface $X_{A^T} \subseteq W\PP^n(r_0, \ldots, r_n)$. Define the diagonal automorphism group, $\aut(F_{A^T})$, analogously to $\aut(F_A)$.  By the analogous isomorphism to that in Equation \ref{autisom}, the group $\aut(F_{A^T})$ is generated by $\rho_i^T := \diag(e^{2\pi i b_{ij}/d})_{j=0}^n \in (\CC^*)^{n+1}$. Define the dual group $G^T$ relative to $G$ to be
\begin{equation}
G^T : = \left\{\prod_{i=0}^n (\rho_i^T)^{s_i} \middle| s_i \in \ZZ, \text{ where }\prod_{i=0}^n x_i^{s_i} \text{ is $G$-invariant}\right\} \subseteq \aut(F_{A^T}).
\end{equation}

\begin{lem} If the group $G$ is a subgroup of $SL(F_A)$, then the dual group $G^T$ contains the group $J_{F_{A^T}}$.  
\end{lem}
\begin{proof}
It is sufficient to show that the element $\prod_{j=0}^n \rho_j^T$ is in the dual group $G^T$. This is equivalent to $\prod_{j=0}^n x_j$ to be $G$-invariant.  Any element $(\lambda_0, \ldots, \lambda_n)$ of $G$ acts on the monomial $\prod_{j=0}^n x_j$ by $\prod_{j=0}^n \lambda_j = 1$ (as $G\subseteq SL(F_A)$).
\end{proof}

\begin{lem}
If the group $G$ contains $J_{F_A}$, then the dual group $G^T$ is contained in $SL(F_{A^T})$.
\end{lem}
The proof of this lemma is analogous to the lemma above.  As the dual group $G^T$ sits between $J_{F_{A^T}}$ and $SL(F_{A^T})$, define the group $\tilde G^T : = G^T/ J_{F_{A^T}}$.  We have a well-defined Calabi-Yau orbifold $Z_{A^T, G^T} := X_{A^T}/ \tilde G^T \subset W\PP^n(r_0, \ldots, r_n)/\tilde G^T$. The Calabi-Yau orbifold $Z_{A^T, G^T}$ is the BHK mirror to $Z_{A,G}$.
 
 \subsection{Classical Mirror Symmetry for BHK Mirrors}  In this section, we summarize some results of Chiodo and Ruan for BHK mirrors. This section is based on Section 3.2 of \cite{CR}. We recommend the exposition there.  Recall that we can view the weighted projective $n$-space $W\PP^n(q_0,\ldots, q_n)$ as a stack
\begin{equation}
\left[ \CC^{n+1} \setminus \{0\} / \CC^*\right]
\end{equation}
where a group element $\lambda$ of the torus $\CC^*$ acts by 
\begin{equation}
\lambda \cdot (x_0, \ldots, x_n) = (\lambda^{q_0} x_0, \ldots, \lambda^{q_n} x_n).
\end{equation}
The quotient stack $W\PP^n(q_0, \ldots, q_n)/ \tilde G$ is equivalent to the stack
\begin{equation}
\left[\CC^{n+1}\setminus\{0\} / G\CC^*\right]
\end{equation}
 so we can view the Calabi-Yau orbifold $Z_{A, G}$ as the (smooth) DM stack
 \begin{equation}
 [Z_{A,G}] : = \left[ \left\{ x \in \CC^{n+1} \setminus \{0\} | F_A(x) = 0\right\} /G\CC^*\right] \subseteq \left[\CC^{n+1}\setminus\{0\} / G\CC^*\right].
 \end{equation}
We now review the Chen-Ruan orbifold cohomology for such a stack. Intuitively speaking, it consists of a direct sum over all elements of $G\CC^*$ of $G$-invariant cohomology of the fixed loci of each element.

If $\gamma$ is an element of $G\CC^*$, take the fixed loci
 \begin{equation}\begin{aligned}
 \CC_\gamma^{n+1}:&= \{ \mathbf{x} \in \CC^{n+1}\setminus \{0\} | \gamma \cdot\mathbf{x} = \mathbf{x}\};\text{ and }  \\ X_A^\gamma :&= \{ F_A|_{\CC_\gamma^{n+1}} = 0\} \subset \CC_\gamma^{n+1}.
 \end{aligned}\end{equation}

 Fix a point $\mathbf{x} \in X_A^\gamma$. The action of $\gamma$ on the tangent space $T_{\mathbf{x}}(\{F_A = 0\})$ can be written as a diagonal matrix (when written with respect to a certain basis), $\Lambda_\gamma = \diag(e^{2\pi i a_1^\gamma}, \ldots, e^{2\pi i a_{n}^\gamma})$, for some real numbers $a_i^\gamma \in [0,1)$.  We then define the \emph{age shift} of $\gamma$, 
\begin{equation}
a(\gamma) := \frac{1}{2\pi i} \log(\det \Lambda_\gamma) = \sum_{j=1}^n a_j^\gamma.
\end{equation}
We now may define the bigraded Chen-Ruan orbifold cohomology as a direct sum of twisted sector ordinary cohomology groups:
\begin{equation}
H^{p,q}_{\text{CR}} ( [Z_{A,G}], \CC) = \bigoplus_{\gamma \in G\CC^*} H^{p-a(\gamma), q-a(\gamma)}( X_A^\gamma / G\CC^*, \CC).
\end{equation}
The degree $d$ Chen-Ruan orbifold cohomology is defined to be the direct sum
\begin{equation}
H^d_{\text{CR}} \left( [Z_{A,G}], \CC\right)  = \bigoplus_{p+q=d} H^{p,q}_{\text{CR}} ( [Z_{A,G}], \CC) .
\end{equation}
 
Continue to assume that the group $G$ contains $J_{F_A}$ and is a subgroup of $SL(F_A)$ and the hypersurface $X_A$ is Calabi-Yau. Chiodo and Ruan prove:

\begin{thm}[Theorem 2 of \cite{CR}]
Given the Calabi-Yau orbifold $Z_{A,G}$ and its BHK mirror $Z_{A^T, G^T}$ as above, one has the standard relationship between the Hodge diamonds of mirror pairs on the level of the Chen-Ruan cohomology of the orbifolds:
\[
H^{p,q}_{\text{CR}} ( [Z_{A,G}], \CC) \cong H^{n-1-p,q}_{\text{CR}}([Z_{A^T,G^T}], \CC).
\]
\end{thm}

This is a classical mirror symmetry theorem for such orbifolds.  We remark that in the case of orbifolds the dimension of the bigraded Chen-Ruan orbifold cohomology vector spaces and stringy Hodge numbers agree.  Moreover, we have:

\begin{cor}[Corollary 4 of \cite{CR}]
Suppose both Calabi-Yau orbifolds $Z_{A,G}$ and $Z_{A^T,G^T}$ admit smooth crepant resolutions $M$ and $W$ respectively, then we have the equality
$$
h^{p,q}(M, \CC) = h^{n-1-p,q}(W, \CC),
$$
where $h^{p,q}$ is the ordinary $(p,q)$ Hodge number.
\end{cor}

\section{The Shioda Map and BHK Mirrors}

We now introduce the Shioda map and relate it to BHK mirrors.  Recall the hypersurfaces $X_A$ and $X_{A^T}$ as above. Define the matrix $B$ to be $ dA^{-1}$ where $d$ is a positive integer so that $B$ has only integer entries. The Shioda maps are the rational maps
\begin{equation}\begin{aligned}
\phi_B: \PP^n \dashrightarrow W\PP^n(q_0,\ldots, q_n),\text{ and } \\
\phi_{B^T}: \PP^n \dashrightarrow W\PP^n(r_0,\ldots, r_n), 
\end{aligned}\end{equation}
where 
\begin{equation}\begin{aligned}
(y_0: \ldots:y_n) &\stackrel{\phi_B}{\mapsto} (x_0:\ldots: x_n), \quad x_j = \prod_{k=0}^n y_k^{b_{jk}}, \text{ and }\\
(y_0: \ldots:y_n) &\stackrel{\phi_{B^T}}{\mapsto} (X_0:\ldots: X_n), \quad X_j = \prod_{k=0}^n y_k^{b_{kj}}.
\end{aligned}\end{equation}

Consider the polynomial 
\begin{equation}
F_{dI} : = \sum_{i=0}^n y_i^d
\end{equation}
and the Fermat hypersurface cut out by it, $X_{dI} := Z(F_{dI}) \subset \PP^n$. Note that the Shioda maps above restrict to rational maps $X_{dI} \stackrel{\phi_B}{\dashrightarrow} X_A$ and $X_{dI} \stackrel{\phi_B}{\dashrightarrow} X_{A^T}$, respectively, allowing us to obtain the diagram: 
\begin{equation}
\xymatrix{ & X_{dI} \ar@{-->}[dl]_{\phi_B} \ar@{-->}[dr]^{\phi_{B^T}} & \\
		X_A && X_{A^T} }
\end{equation}

We now reinterpret the groups $G$ and $G^T$ in the context of the Shioda map. Any element of $\aut(F_{dI})$ is of the form $g=(e^{2\pi i h_j/d})_j$, for some integers $h_j$. When we push forward the action of $g$ via $\phi_B$, we obtain  the diagonal automorphism
\begin{equation}
(\phi_B)_*(g) : = (e^{{\frac{2\pi i }{d} \sum_{j=0}^n b_{ij}h_j}})_i \in \aut(F_A).
\end{equation}
The element $(\phi_B)_*(g)$ is a generic element of $\aut(F_{A})$, namely $\prod_{j=0}^n \rho_j^{h_j} $.  
We turn our attention to describing the dual group $G^T$ to $G$. If we push the element $g^T:= (e^{2\pi i s_i/d})_i \in \aut(F_{dI})$ down via the map $\phi_{B^T}$, then we get the action
\begin{equation}
(\phi_{B^T})_*(g^T) =(e^{\frac{2\pi i }{d} \sum_{i=0}^n s_i b_{ij}})_j= \prod_{i=1}^n (\rho^T_i)^{s_i} .
\end{equation}

 In other words, we have (surjective) group homomorphisms
\begin{equation}\begin{aligned}
(\phi_B)_*&: \aut(F_{dI}) \rightarrow \aut(F_A);\text{ and }\\
(\phi_{B^T})_*&: \aut(F_{dI}) \rightarrow \aut(F_{A^T}).
\end{aligned}\end{equation}

This gives us a new interpretation of the choice of groups $G$ and $G^T$:  both are pushforwards of subgroups of $\aut(F_{dI})$ via the Shioda maps $\phi_B$ and $\phi_{B^T}$, respectively.

\subsection{Reinterpretation of the Dual Group} 
We now reformulate the relationship between the groups $G$ and $G^T$ via a bilinear pairing. Consider the map
\[
\langle , \rangle_B: \ZZ^{n+1} \times \ZZ^{n+1} \longrightarrow \ZZ
\]
where $\langle \mathbf{s}, \mathbf{h}\rangle_B := \mathbf{s}^T B \mathbf{h}$. 
Choose a subgroup $G\subset \aut (F_A)$, so that $J_{F_A} \subseteq G$.  Then set $H : = ((\phi_B)_*)^{-1}(G)$. Note that the map $(h_j)_j \mapsto (e^{2\pi i h_j / d})_j$ induces a natural, surjective group homomorphism
\begin{equation}
\ZZ^{n+1} \stackrel{pr}{\rightarrow} \aut(F_{dI}).
\end{equation}
Take $\tilde H$ to be the inverse image $\tilde H:= pr^{-1}(H)$ of $H$ under this map. We can then define the subgroup $\tilde H^{\perp_B} \subseteq \ZZ^{n+1}$ to be 
\begin{equation}
\tilde H^{\perp_B} := \left\{ \mathbf{s} \in \ZZ^{n+1} \middle| \langle \mathbf{s}, \mathbf{h}\rangle_B \in d\ZZ \text{ for all } \mathbf{h} \in \tilde H\right\}.
\end{equation}
Define $H^{\perp_B}$ to be the image of $\tilde H^{\perp_B}$ under $pr$, $pr(\tilde H^{\perp_B})$.  

We remark that it is clear that the group $J_{F_{dI}}$ is contained by $H$ as $(\phi_B)_*(e^{2\pi i / d}, \ldots, e^{2\pi i/d}) = \prod_j \rho_j$ is a generator of $ J_{F_A}$.

We have assumed that the group $G$ is a subgroup of $SL(F_A)$. This requires that, for all group elements $\mathbf{h} = (h_k)_k \in \tilde H$, the product $\prod_{j=0}^ne^{\frac{2\pi i }{d} \sum_{k} b_{jk} h_k}$ equals $1$.  This implies that the sum  $\sum_{j,k = 0}^n b_{jk}h_k$ is an integer divisible by $d$; therefore, $(1,\ldots, 1) \in \tilde H^{\perp_B}$. So, its image $pr(1,\ldots, 1)$ must be in $H^{\perp_B}$.  The element $pr(1,\ldots, 1) = (e^{2\pi i/d}, \ldots, e^{2\pi i/d})$ is a generator of the group $J_{F_{dI}}$, hence $H^{\perp_B}$ contains $J_{F_{dI}}$.   

Moreover, if one unravels all the definitions, one can see that $(\phi_{B^T})_*(H^{\perp_B}) = G^T$. In order for a monomial $\prod_{i=0}^n x_i^{s_i}$ to be $G$-invariant, we will need, for any $\prod_{i=1}^n \rho_i^{h_i} =(e^{{\frac{2\pi i }{d} \sum_{i=0}^n b_{ij}h_j}})_i \in G$, that $\prod_{i=0}^n x_i^{s_i} = \prod_{i=0}^n (e^{{\frac{2\pi i }{d} \sum_{i=0}^n b_{ij}h_j}} x_i)^{s_i}$. This is equivalent to $\sum_{i,j} s_ib_{ij}h_j$ being a multiple of $d$. 

\subsection{Birational Geometry of BHK Mirrors}
We now give a Theorem of Bini, written in our notation (Theorem 3.1 of \cite{Bi}):
\begin{thm}
Let all the notation be as above. Then the hypersurfaces $X_A$ and $X_{A^T}$ are birational to the quotients of the Fermat variety $X_{dI}/ (((\phi_B)_*)^{-1}(J_{F_A})/J_{F_{dI}})$ and $X_{dI}/ (((\phi_{B^T})_*)^{-1}(J_{F_{A^T}})/J_{F_{dI}})$, respectively.
\end{thm}

We now give a few comments about the proof of the above theorem.  It is proven via composing $\phi_B$ with the map 
\begin{equation}\begin{aligned}
q_A : W\PP^n(q_0, \ldots, q_n) &\dashrightarrow \PP^{n+1};\\
(x_0:\ldots: x_n) &\longmapsto \left(\prod_j x_j : \prod_j x_j^{a_{1j}}:\ldots:\prod_j x_j^{a_{nj}}\right).
\end{aligned}\end{equation}
Note that the composition $q_A \circ \phi_B: X_{dI} \dashrightarrow \PP^{n+1}$ gives the map
\begin{equation}
(y_0:\ldots:y_n) \longmapsto \left(\prod_j y_j^{q_j'}: y_1^d:\ldots: y_n^d\right).
\end{equation}
Letting $m = \gcd(d, q_1', \ldots, q_n')$, we describe the closure of the image as \begin{equation}
\overline{M_A} := Z\left(\sum_{i=1}^n u_i, u_0^{d/m} = \prod_{i=1}^n u_i^{q_i'/m}\right) \subset \PP^{n+1}.
\end{equation}
  Bini then proves that the map $q_A\circ\phi_B$ is birational to a quotient map, which in our notation implies the birational equivalence
\begin{equation}
\overline{M_A} \simeq X_{dI} / (\phi_B^{-1}(SL(F_A))/ J_{F_{dI}}).
\end{equation}

Bini then refers the reader to the proof of Theorem 2.6 in \cite{BvGK} to see why the other two maps are birational to quotient maps as well. Note that Bini requires $d$ to be the smallest positive integer so that $dA^{-1}$ is an integral matrix, but the requirement that $d$ is the smallest such integer is unnecessary. One can just use the first part of Theorem 2.6 of \cite{BvGK} to eliminate this hypothesis.

An upshot of this reinterpretation of the theorem is that the mirror statement of BHK duality is a relation of two orbifolds birational to different orbifold quotients of the same Fermat hypersurface in projective space. Namely, $X_A/ \tilde G$ is birational to $X_{dI}/ ( ((\phi_B)_*)^{-1}(J_{F_A})+ H/J_{F_{dI}})$ while $X_{A^T}/\tilde G^T$ is birational to $X_{dI}/ (((\phi_{B^T})_*)^{-1}(J_{F_{A^T}})+ H^{\perp_B}/J_{F_{dI}})$.  As $J_{F_A} \subseteq G$ and $J_{F_{A^T}} \subseteq G^T$, 

\begin{equation}\begin{aligned}
((\phi_B)_*)^{-1}(J_{F_A}) \subseteq ((\phi_B)_*)^{-1}(G) = H; \text{ and } \\((\phi_{B^T})_*)^{-1}(J_{F_{A^T}}) \subseteq ((\phi_{B^T})_*)^{-1}(G^T) \subseteq H^{\perp_B}
\end{aligned}\end{equation}
 which gives us the following corollary:

\begin{cor}\label{binicor}
The Calabi-Yau orbifold $Z_{A,G}$ is birational to $X_{dI}/ ( H/J_{F_{dI}})$ and its BHK mirror $Z_{A^T,G^T}$  is birational to $X_{dI}/ (H^{\perp_B}/J_{F_{dI}})$.
\end{cor}

\section{Multiple Mirrors}

As stated in the introduction, one can take two polynomials 
\begin{equation}
F_A = \sum_{i=0}^n \prod_{j=0}^n x_j^{a_{ij}}; \qquad F_{A'} = \sum_{i=0}^n \prod_{j=0}^n x_j^{a_{ij}'}.
\end{equation}
that cut out two hypersurfaces in weighted-projective $n$-spaces, $X_A \subseteq W\PP^n(q_0,\ldots, q_n)$ and $X_{A^T} \subseteq W\PP^n(q_0',\ldots, q_0')$, respectively.  Take two groups $G$ and $G'$ so that $J_{F_A} \subseteq G \subseteq SL(F_A)$ and $J_{F_{A^T}} \subseteq G^T \subseteq SL(F_{A^T})$. We then obtain two Calabi-Yau orbifolds $Z_{A,G} : = X_A/\tilde G$ and $Z_{A', G'} : = X_{A'} / \tilde G '$.  

Even if these two orbifolds are in the same family of Calabi-Yau varieties, they may have BHK mirrors that are not in the same quotient of weighted-projective $n$-space (see Section 5 for an explicit example or Tables 5.1-3 of \cite{DG}).  Take the polynomials
\begin{equation}
F_{A^T} = \sum_{i=0}^n \prod_{j=0}^n x_j^{a_{ji}}; \qquad F_{(A')^T} = \sum_{i=0}^n \prod_{j=0}^n x_j^{a_{ji}'}.
\end{equation}
They are quasihomogeneous polynomials but not necessarily with the same weights. So they cut out hypersurfaces $X_{A^T}$ and $X_{(A')^T}$. Take the dual groups $G^T$ and $(G')^T$ to $G$ and $G'$, respectively. We quotient each hypersurface by their respective groups, $\tilde G^T := G^T/ J_{F_{A^T}}$ and $(\tilde G')^T : = (G')^T / J_{F_{(A')^T}}$. We then have the BHK mirror dualities:
\begin{equation}\begin{aligned}
Z_{A,G} &\stackrel{\text{BHK mirrors}}{\longleftrightarrow} Z_{A^T, G^T} \\
Z_{A',G'} &\stackrel{\text{BHK mirrors}}{\longleftrightarrow} Z_{(A')^T, (G')^T} 
\end{aligned}\end{equation}

In this section, we will investigate and compare the birational geometry of the BHK mirrors of the Calabi-Yau orbifolds $Z_{A, G}$ and $Z_{A',G'}$ by using the Shioda maps.  Take positive integers $d$ and $d'$ so that $B:= d A^{-1}$ and $B':= d' (A')^{-1}$ are matrices with integer entries.  Then we can form a ``tree'' diagram of Shioda maps: 
\begin{equation}
\xymatrix{
	& 	&		&	X_{dd'I} \ar@{-->}[dll]_{\phi_{d'I}} \ar@{-->}[drr]^{\phi_{dI}} &	&	& \\   &X_{dI} \ar@{-->}[dl]_{\phi_B} \ar@{-->}[dr]^{\phi_{B^T}}&&&&X_{d'I}\ar@{-->}[dl]_{\phi_{B'}} \ar@{-->}[dr]^{\phi_{(B')^T}}&\\
X_A &	& X_{A^T} & 							                      & X_{A'}&	&X_{(A')^T}
		}
\end{equation}

One can then calculate that $\phi_M \circ \phi_{cI} = \phi_{cM}$ for any integer valued matrix $M$ and positive integer $c$.   This means we can simplify our tree to just the diagram:
\begin{equation}
\xymatrix{
	& 	&		&	X_{dd'I} \ar@{-->}[dlll]_{\phi_{d'B}} \ar@{-->}[dl]^{\phi_{d'B^T}} \ar@{-->}[dr]_{\phi_{dB'}} \ar@{-->}[drrr]^{\phi_{d(B')^T}} &	&	& \\   
X_A &	& X_{A^T} & 							                      & X_{A'}&	&X_{(A')^T}
		}
\end{equation}

The Calabi-Yau orbifolds are just finite quotients of the hypersurfaces $X_A$, $X_{A^T}$, $X_{A'}$ and $X_{(A')^T}$, so we can view them in the context of the diagram:
\begin{equation}
\xymatrix{
	& 	&		&	X_{dd'I} \ar@{-->}[dlll]_{\phi_{d'B}} \ar@{-->}[dl]^{\phi_{d'B^T}} \ar@{-->}[dr]_{\phi_{dB'}} \ar@{-->}[drrr]^{\phi_{d(B')^T}} &	&	& \\   
X_A \ar[d]&	& X_{A^T}\ar[d] & 							                      & X_{A'}\ar[d]&	&X_{(A')^T}\ar[d] \\
Z_{A,G} && Z_{A^T, G^T} && Z_{A', G'} && Z_{(A')^T, (G')^T}
		}
\end{equation}
Letting $H$ and $H'$ be the groups $H:= (\phi_{d'B})_*^{-1}(G)$ and $H':= (\phi_{dB'})_*^{-1}(G')$, we know that:

\begin{prop}\label{biratprop}
The following birational equivalences hold:
\begin{equation}\begin{aligned}
Z_{A,G} &\simeq X_{dd'I}/(H/J_{F_{dd'I}}); \\
Z_{A^T, G^T} & \simeq X_{dd'I}/ (H^{\perp_{d'B}}/J_{F_{dd'I}}); \\
Z_{A', G'} &\simeq X_{dd'I}/(H'/J_{F_{dd'I}}); \text{ and } \\
Z_{(A')^T, (G')^T} &\simeq X_{dd'I}/ ((H')^{\perp_{dB'}}/J_{F_{dd'I}}).
\end{aligned}\end{equation}
\end{prop}
\begin{proof}
Follows directly from Corollary \ref{binicor}.
\end{proof}
Recall that we are asking for the conditions in which $Z_{A^T, G^T}$ and $Z_{(A'), (G')^T}$ are birational. This question can be answered if we can show that the groups $H^{\perp_{d'B}}$ and $(H')^{\perp_{dB'}}$ are equal.  We now prove that such an equality holds, if we assume that the groups $G$ and $G'$ are equal.

\begin{lem}
If the groups $G$ and $G'$ are equal, then $H^{\perp_{d'B}}$ and $(H')^{\perp_{dB'}}$ are equal.
\end{lem}
\begin{proof}

Set $\tilde H : = pr^{-1}(H) $ and $\tilde H' : = pr^{-1}(H')$ (Recall these groups from Section 3.2). Note that we have an equality of groups  $( \phi_{d'B})_*\circ pr(\tilde H ) = G = G ' = (\phi_{dB'})_*\circ pr(\tilde H')$.  This implies that, for any element $\mathbf{h} \in \tilde H$, there exists an element $\mathbf{h'} \in \tilde H'$ so that $d'B\mathbf{h} = dB' \mathbf{h}'$.   

Suppose that $\mathbf{s} \in (\tilde H')^{\perp_{dB'}} $. We claim that $\mathbf{s}$ is in $\tilde H ^{\perp_{d'B}}$, i.e., for every $\mathbf{h} \in \tilde H$, that $\langle \mathbf{s}, \mathbf{h} \rangle_{d'B} \in d \ZZ$. Indeed, this is true. Given any $\mathbf{h} \in \tilde H$, there exists some $\mathbf{h'}$ as above where $d'B\mathbf{h} = dB'\mathbf{h'}$, hence  $\langle \mathbf{s} , \mathbf{h}\rangle_{d'B} = \langle \mathbf{s}, \mathbf{h'}\rangle_{dB'} \in d\ZZ$, as $\mathbf{s} \in (\tilde H')^{\perp_{dB'}} $.  This proves that $(\tilde H)^{\perp_{d'B}} \subseteq (\tilde H')^{\perp_{dB'}}$. By symmetry, we now have the equality of the groups, $\tilde H^{\perp_{d'B}} = (\tilde H')^{\perp_{dB'}}$. 

This implies that the images of the groups  $\tilde H^{\perp_{d'B}}$ and $ (\tilde H')^{\perp_{dB'}}$ under the homomorphism $pr$ are equal, hence $H^{\perp_{d'B}}$ and $(H')^{\perp_{dB'}}$ are equal.
\end{proof}

We then have the proof of Theorem \ref{main}:
\begin{thm}[=Theorem \ref{main}]
Let $Z_{A,G}$ and $Z_{A',G'}$ be Calabi-Yau orbifolds as above. If the groups $G$ and $G'$ are equal, then the BHK mirrors $Z_{A^T, G^T}$ and $Z_{(A')^T, (G')^T}$ of these orbifolds are birational.
\end{thm}
\begin{proof}
Follows directly from Proposition 4.1 and Lemma 4.2.
\end{proof}


\section{Example: BHK Mirrors, Shioda Maps, and Chen-Ruan Hodge Numbers}
In this section, we give an example of two Calabi-Yau orbifolds $Z_{A,G}$ and $Z_{A', G'}$ that are in the same family, but have two different BHK mirrors $Z_{A^T, G^T}$ and $Z_{(A')^T, (G')^T}$ that are not in the same family.  As mentioned before, this is a feature of BHK mirror duality that differentiates it from the mirror construction of Batyrev and Borisov.  We will show that the BHK mirrors are birational to each other and that their Chen-Ruan Hodge numbers match.
 
Consider the polynomials 
\begin{equation}\begin{aligned}
F_A : &= y_1^8 + y_2^8 + y_3^4 + y_4^3 +y_5^6; \text{ and }\\
F_{A'} :&= y_1^8 + y_2^8 + y_3^4 + y_4^3 + y_4 y_5^4.
\end{aligned}\end{equation}

The polynomials $F_A$ and $F_{A'}$ cut out hypersurfaces $X_A = Z(F_A)$ and $X_{A'} = Z(F_{A'})$, two well-defined hypersurfaces in the (Gorenstein) weighted projective 4-space $W\PP^4(3, 3, 6, 8, 4)$.  Note that they are in the same family.

We now address the groups involved in the BHK mirror construction. Set $\zeta$ to be a primitive 24th root of unity. The groups $J_{F_A}$ and $J_{F_{A'}}$ are equal and are generated by the element $ (\zeta^3, \zeta^3, \zeta^6, \zeta^8, \zeta^4) \in (\CC^*)^5$. We take $G$ and $G'$ to be the same group, namely we define it to be 
\begin{equation}
G = G' := \langle (\zeta^3, \zeta^3, \zeta^6, \zeta^8, \zeta^4), (\zeta^{18}, 1, \zeta^6, 1, 1), (1,1,\zeta^{12}, 1,\zeta^{12})\rangle.
\end{equation}
Note that each of the generators of the group $G$ are also in $SL(F_A)$ and $SL(F_{A'})$, hence the group $G$ satisfies the conditions required for BHK duality. We quotient both the hypersurfaces by $X_A$ and $X_{A'}$ by the group $G/J_{F_A}$ to obtain the Calabi-Yau orbifolds $Z_{A,G}$ and $Z_{A',G}$ which are in the same family of hypersurfaces in $W\PP^4(3,3,6,8,4)/(G/J_{F_A})$.

\subsection{BHK Mirrors}
Next, we describe the BHK mirrors to $Z_{A,G}$ and $Z_{A',G'}$. The polynomials associated to the matrices $A$ and $A^T$ are
\begin{equation}\begin{aligned}
F_{A^T} = F_A : &= y_1^8 + y_2^8 + y_3^4 + y_4^3 +y_5^6; \text{ and }\\
F_{A'^T} :&= z_1^8 + z_2^8 + z_3^4 + z_4^3z_5 + z_5^4.
\end{aligned}\end{equation}
While the hypersurface $X_{A^T} = Z(F_{A^T})$ is in $W\PP^4(3,3,6,8,4)$, they hypersurface $X_{(A')^T} = Z(F_{A'^T})$ is in a different (Gorenstein) weighted projective 4-space, namely $W\PP^4(1,1,2,2,2)$.  We can compute the following groups: 
\begin{equation}\begin{aligned}
J_{F_{A^T}} &= \langle (\zeta^3, \zeta^3, \zeta^6, \zeta^8,\zeta^4)\rangle;\\
J_{F_{(A')^T}} &=  \langle (\zeta^3, \zeta^3, \zeta^6, \zeta^6,\zeta^6)\rangle; \\
G^T &= \langle (\zeta^3, \zeta^3, \zeta^6, \zeta^8,\zeta^4)\rangle;\text{ and } \\
(G')^T &= \langle (\zeta^3, \zeta^3,\zeta^6,\zeta^6,\zeta^6), (1,1,1, \zeta^{12},\zeta^{12})\rangle.
\end{aligned}\end{equation}

Note that the groups $G^T$ and $J_{F_{A^T}}$ are equal, so the BHK mirror $Z_{A^T,G^T}$ is the hypersurface $X_{A^T}$.  On the other hand the quotient group $(G')^T/J_{F_{(A')^T}}$ is isomorphic to $\ZZ_2$, hence the BHK mirror $Z_{(A')^T, (G')^T}$ is the Calabi-Yau orbifold $X_{(A')^T}/ \ZZ_2$.  Note that the Calabi-Yau orbifold $Z_{A^T,G^T}$ is a hypersurface in $ W\PP^4(3,3,6,8,4)$, while $Z_{(A')^T, (G')^T}$ is in $W\PP(1,1,2,2,2)/\ZZ_2$. The two BHK mirrors are not hypersurfaces of the same quotient of weighted-projective spaces, hence not sitting inside the same family of Calabi-Yau orbifolds.  

\subsection{Shioda Maps}
Even though the two BHK mirrors $Z_{A^T, G^T}$ and $Z_{(A')^T, (G')^T}$ do not sit in the same family of hypersurfaces of the same quotient of weighted-projective space, we can show that they are birational.  Take the matrices $B:= 24A^{-1}$ and $B' : = 24(A')^{-1}$. Let $X_{24I}$ be the hypersurface $Z(x_1^{24} + x_2^{24} +x_3^{24} +x_4^{24} +x_5^{24} )\subset \PP^4$. We then have the Shioda maps 

\begin{equation}
\xymatrix{
	& 	&		&	X_{24I} \ar[dlll]_{\phi_{B}} \ar[dl]^{\phi_{B^T}} \ar@{-->}[dr]_{\phi_{B'}} \ar@{-->}[drrr]^{\phi_{(B')^T}} &	&	& \\   
X_A &	& X_{A^T} & 							                      & X_{A'}&	&X_{(A')^T}
		}
\end{equation}

The maps then can be described explicitly:
\begin{equation}\begin{aligned}
\phi_B(x_1:\ldots: x_5) &= (x_1^3: x_2^3:x_3^6:x_4^8:x_5^4) \in X_A \\
\phi_{B^T}(x_1:\ldots: x_5) &= (x_1^3: x_2^3:x_3^6:x_4^8:x_5^4) \in X_{A^T} \\
\phi_{B'}(x_1:\ldots: x_5) &= (x_1^3: x_2^3:x_3^6:x_4^8:x_4^{-2}x_5^6)\in X_{A'} \\
\phi_{(B')^T} (x_1:\ldots: x_5) &= (x_1^3: x_2^3:x_3^6:x_4^8x_5^{-2}:x_5^6)\in X_{(A')^T} \\
\end{aligned}\end{equation}

The four Shioda maps are rational maps that are birational to quotient maps. Take the following four subgroups to $\aut(F_{24I})$:
\begin{equation}\begin{aligned}
H &:= \langle (\zeta, \zeta,\zeta,\zeta,\zeta), (\zeta^8, 1,1,1,1), (\zeta^2, 1, \zeta^{-1}, 1,1), (1,1,\zeta^2, 1,\zeta^3),(1,1,1,\zeta,\zeta^4)\rangle; \\
H' &:= \langle  (\zeta, \zeta,\zeta,\zeta,\zeta), (\zeta^8, 1,1,1,1), (\zeta^2, 1, \zeta^{-1}, 1,1), (1,1,\zeta^2, 1, \zeta^2), (1,1,1,\zeta^3,\zeta)\rangle; \\
H^{\perp_B} = (H')^{\perp_{B'}}& :=\langle  (\zeta, \zeta,\zeta,\zeta,\zeta), (\zeta^8, 1,1,1,1), (1,1,\zeta^4, 1,1), (\zeta^2, \zeta^2, \zeta^2, 1,1), (1,1,1,\zeta,\zeta^4)\rangle; \\
J_{F_{24I}} &= \langle (\zeta,\zeta,\zeta,\zeta,\zeta)\rangle.
\end{aligned}\end{equation}

By Proposition \ref{biratprop}, we have the following birational equivalences
\begin{equation}\begin{aligned}
Z_{A,G} &\simeq X_{24I}/\langle (\zeta^8, 1,1,1,1), (\zeta^2, 1, \zeta^{-1}, 1,1), (1,1,\zeta^2, 1,\zeta^3),(1,1,1,\zeta,\zeta^4)\rangle; \\
Z_{A', G'} &\simeq X_{24I}/\langle (\zeta^8, 1,1,1,1), (\zeta^2, 1, \zeta^{-1}, 1,1), (1,1,\zeta^2, 1, \zeta^2), (1,1,1,\zeta^3,\zeta)\rangle; \\
Z_{A^T, G^T} & \simeq X_{24I}/ \langle  (\zeta^8, 1,1,1,1), (1,1,\zeta^4, 1,1), (\zeta^2, \zeta^2, \zeta^2, 1,1), (1,1,1,\zeta,\zeta^4)\rangle;\text{ and }  \\
Z_{(A')^T, (G')^T} &\simeq X_{24I}/ \langle  (\zeta^8, 1,1,1,1), (1,1,\zeta^4, 1,1), (\zeta^2, \zeta^2, \zeta^2, 1,1), (1,1,1,\zeta,\zeta^4)\rangle.
\end{aligned}\end{equation}
So we can see that the BHK mirrors $Z_{A^T, G^T}$ and $Z_{(A')^T, (G')^T}$ are birational.
\subsection{Chen-Ruan Hodge Numbers}

As the Calabi-Yau orbifolds $Z_{A,G}$ and $Z_{A',G}$ are quasismooth varieties in the same toric variety, namely $W\PP^4(3,3,6,8,4)/\langle (-i,1,i,1,1), (1,1,-1,1,-1)\rangle$, they have the same Chen-Ruan Hodge numbers. By the theorem of Chiodo and Ruan, this means that their BHK mirrors $Z_{A^T, G^T}$ and $Z_{(A')^T, (G')^T}$ must have the same Chen-Ruan Hodge numbers. We now check this explicitly.

Consider the hypersurface $X_{A^T} \subseteq W\PP^4(3,3,6,8,4)$. The dual group $G^T$ is equal to the group $J_{F_{A^T}}$. The only elements of the group $G^T\CC^*$ that will have nontrivial fixed loci are in $J_{F_{A^T}}$ as the weighted projective space is Gorenstein. The group $J_{F_{A^T}}$ has exactly six elements which have fixed loci that have nonempty intersections with the hypersurface: 

\begin{center}
  \begin{tabular}{ c | c }
    
    Element of $J_{F_{A^T}}$ & Fixed Locus \\ \hline
    $(1,1,1,1,1)$ & $X_{A^T}$ \\ \hline
    $(\zeta^{18}, \zeta^{18}, \zeta^{12}, 1,1)$ & 	$Z(y_1,y_2,y_3)\cap X_{A^T}$	 \\   \hline
    $(1,1,1, \zeta^{16}, \zeta^8)$ & $Z(y_4,y_5)\cap X_{A^T}$ \\ \hline
    $(\zeta^{12}, \zeta^{12},1,1,1)$ & $Z(y_1, y_2) \cap X_{A^T}$ \\ \hline
    $(1,1,1, \zeta^{8}, \zeta^{16})$ & $Z(y_4,y_5)\cap X_{A^T}$ \\ \hline 
    $(\zeta^{6}, \zeta^{6}, \zeta^{12}, 1,1)$ & $Z(y_1,y_2,y_3)\cap X_{A^T}$ 
  \end{tabular}
\end{center}

We can just then compute the Hodge numbers by using the Griffiths-Dolgachev-Steenbrink formulas (see \cite{Do}).  This computation gives us that $X_{A^T}$ has a Hodge diamond of: 
\[ \begin{array}{ccccccc}
 &&  & 1 &  && \\
& & 0 &  & 0 && \\
&0 &  & 1 &  & 0&\\
 1&& 36 &  & 36 & &1\\
 &0 &  & 1 &  & 0&\\
 & & 0 &  & 0 && \\
  &&  & 1 &  && \\\end{array}
 \]
The remaining fixed loci are simpler: $Z(y_1,y_2,y_3)\cap X_{A^T}$ consists of three points, $Z(y_4,y_5)\cap X_{A^T}$ is a curve of genus nine, and $Z(y_1, y_2) \cap X_{A^T}$ is a curve of genus one. After considering the age shift, one obtains the Chen-Ruan Hodge diamond of the Calabi-Yau orbifold $Z_{A,G}$:
\[ \begin{array}{ccccccc}
 &&  & 1 &  && \\
& & 0 &  & 0 && \\
&0 &  & 7 &  & 0&\\
 1&& 55 &  & 55 & &1\\
 &0 &  & 7 &  & 0&\\
 & & 0 &  & 0 && \\
  &&  & 1 &  && \\\end{array}
 \]

Next, we check that these are the same Chen-Ruan Hodge numbers as the Calabi-Yau orbifold $Z_{(A')^T, (G')^T}$. Recall that $X_{A'^T} \subset W\PP^4(1,1,2,2,2)$, so we will have a different $\CC^*$ action. The group $(G')^T$ equals the group $ J_{F_{(A')^T}}\cdot \langle (1,1,1,-1,-1)\rangle$. As the weighted-projective space is Gorenstein, we can only look at $(G')^T$ to find the nontrivial fixed loci of elements. The group $(G')^T$ only has five elements that will have nonempty intersections between the hypersurface and the fixed loci of the elements:

\begin{center}
  \begin{tabular}{ c | c }
    
    Element of $(G')^T $ & Fixed Locus \\ \hline
     $(1,1,1,1,1)$ &  $X_{(A')^T}$ \\ \hline
     $(\zeta^{12},\zeta^{12},1,1,1)$ & $Z(z_1,z_2)\cap X_{(A')^T}$ \\ \hline
     $(\zeta^6,\zeta^6, \zeta^{12},1,1)$ & $Z(z_4,z_5) \cap X_{(A')^T}$ \\ \hline
     $(\zeta^6,\zeta^6, \zeta^{12},1,1)$ & $Z(z_1,z_2,z_3) \cap X_{(A')^T}$ \\ \hline
     $(\zeta^{18},\zeta^{18}, \zeta^{12},1,1)$ & $Z(z_1,z_2,z_3) \cap X_{(A')^T}$
  \end{tabular}
\end{center}

One then computes the cohomology of each fixed locus and  finds the piece invariant under the action of the group $\ZZ_2$ generated by $(1,1,1,-1,-1)$. The ($\ZZ_2$)-invariant part of the cohomology of $X_{(A')^T}$ gives the Hodge diamond:
\[ \begin{array}{ccccccc}
 &&  & 1 &  && \\
& & 0 &  & 0 && \\
&0 &  & 1 &  & 0&\\
 1&& 45 &  & 45 & &1\\
 &0 &  & 1 &  & 0&\\
 & & 0 &  & 0 && \\
  &&  & 1 &  && \\\end{array}
 \]
 $Z(z_1,z_2)\cap X_{(A')^T}$ is a curve with a $\ZZ_2$ invariant $h^{0,1}=1$, $Z(z_4,z_5) \cap X_{(A')^T}$ is a $\ZZ_2$-invariant curve of genus nine, and $Z(z_1,z_2,z_3) \cap X_{(A')^T}$ is a set of four $\ZZ_2$-invariant points.  After considering the age shift, one obtains the Chen-Ruan Hodge diamond of $Z_{(A')^T, (G')^T}$: 
\[ \begin{array}{ccccccc}
 &&  & 1 &  && \\
& & 0 &  & 0 && \\
&0 &  & 7 &  & 0&\\
 1&& 55 &  & 55 & &1\\
 &0 &  & 7 &  & 0&\\
 & & 0 &  & 0 && \\
  &&  & 1 &  && \\\end{array}
 \]
 
Note that this Chen-Ruan Hodge diamond matches that of the Calabi-Yau orbifold $Z_{A,G}$.  To summarize, what we have given here is two Calabi-Yau orbifolds $Z_{A,G}$ and $Z_{A', G'}$ that live in a family of hypersurfaces in a finite quotient of a weighted-projective space. Their BHK mirrors $Z_{A^T, G^T}$ and $Z_{(A')^T, (G')^T}$ do not sit in a single family, unlike the mirrors proposed by Batyrev and Borisov. However, the two BHK mirrors have the same Chen-Ruan Hodge number and are birationally equivalent to one another, as both are birational to the same finite quotient of a Fermat hypersurface of $\PP^4$. 

 \bibliographystyle{amsplain}

\end{document}